\documentclass[
11pt,  reqno]{amsart}
\usepackage[margin=1.2in,marginparwidth=1.5cm, marginparsep=0.5cm]{geometry}

\usepackage{amsmath,amssymb,amscd,amsthm,amsxtra, esint}

\usepackage{booktabs} 
\usepackage{microtype}
\usepackage{amssymb}
\usepackage{mathrsfs}

\usepackage{upgreek}

\usepackage{color}
\usepackage[implicit=true]{hyperref}

\usepackage{xsavebox}

\usepackage{cases}

\allowdisplaybreaks[2]

\definecolor{gr}{rgb}   {0.,   0.69,   0.23 }
\definecolor{bl}{rgb}   {0.,   0.5,   1. }
\definecolor{mg}{rgb}   {0.85,  0.,    0.85}
\definecolor{yl}{rgb}   {0.8,  0.7,   0.}
\definecolor{or}{rgb}  {0.7,0.2,0.2}

\usepackage{tikz}
\usepackage{marginnote}
\usepackage{scalerel} 

\usetikzlibrary{shapes.misc}
\usetikzlibrary{shapes.symbols}
\usetikzlibrary{decorations}
\usetikzlibrary{decorations.markings}

\tikzset{
	dot/.style={circle,fill=black,draw=black,inner sep=0pt,minimum size=0.5mm},
	>=stealth,
	}

\tikzset{
	dot2/.style={circle,fill=black,draw=black,inner sep=0pt,minimum size=0.2mm},
	>=stealth,
	}

\tikzset{
	ddot/.style={circle,fill=black,draw=black,inner sep=0pt,minimum size=0.8mm},
	>=stealth,
	}

\tikzset{decision/.style={ 
        draw,
        diamond,
        aspect=1.5
    }}

\tikzset{dia2/.style
={diamond,fill=white,draw=black,inner sep=0pt,minimum size=1mm},
	>=stealth,
	}

\tikzset{dia/.style
={star,fill=black,draw=black,inner sep=0pt,minimum size=1mm},
	>=stealth,
	}

\tikzset{dia/.style
={diamond,fill=black,draw=black,inner sep=0pt,minimum size=1.3mm},
	>=stealth,
	}

\makeatletter
\def\DeclareSymbol#1#2#3{\xsavebox{#1}{\tikz[baseline=#2,scale=0.15]{#3}}}
\def\<#1>{\xusebox{#1}}
\makeatother

\newsavebox{\peA}
\newsavebox{\pneA}
\newsavebox{\plA}
\newsavebox{\pgA}
\newsavebox{\pleA}
\newsavebox{\pgeA}
\newsavebox{\pezA}

\savebox{\peA}{\tikz \draw (0,0) node[shape=circle,draw,inner sep=0pt,minimum size=8.5pt] {\scriptsize  $=$};}
\savebox{\pneA}{\tikz \draw (0,0) node[shape=circle,draw,inner sep=0pt,minimum size=8.5pt] {\footnotesize $\neq$};}
\savebox{\plA}{\tikz \draw (0,0) node[shape=circle,draw,inner sep=0pt,minimum size=8.5pt] {\scriptsize $<$};}
\savebox{\pgA}{\tikz \draw (0,0) node[shape=circle,draw,inner sep=0pt,minimum size=8.5pt] {\scriptsize $>$};}
\savebox{\pleA}{\tikz \draw (0,0) node[shape=circle,draw,inner sep=0pt,minimum size=8.5pt] {\scriptsize $\leqslant$};}
\savebox{\pgeA}{\tikz \draw (0,0) node[shape=circle,draw,inner sep=0pt,minimum size=8.5pt] {\scriptsize $\geqslant$};}
\savebox{\pezA}{\tikz \draw (0,0) node[shape=circle,draw,
fill=white, 
inner sep=0pt,minimum size=8.5pt]{} ;}

\def \peB{\mathchoice
{\scalebox{.7}{{\usebox{\peA}}}}
{\scalebox{.7}{{\usebox{\peA}}}}
{\scalebox{.7}{{\usebox{\peA}}}}
{}
}

\def \pezB{\mathchoice
{\scalebox{.7}{{\usebox{\pezA}}}}
{\scalebox{.7}{{\usebox{\pezA}}}}
{\scalebox{.7}{{\usebox{\pezA}}}}
{}
}

\newcommand{\pe}{\mathbin{{\peB}}}

\newcommand{\pez}{\mathbin{{\pezB}}}

\usetikzlibrary{decorations.pathmorphing, patterns,shapes}

\DeclareSymbol{dot}{-2.7}{\draw (0,0) node[dot]{};}

\DeclareSymbol{1}{0}{\draw[white] (-.4,0) -- (.4,0); \draw (0,0)  -- (0,1.2) node[dot] {};}
\DeclareSymbol{2}{0}{\draw (-0.5,1.2) node[dot] {} -- (0,0) -- (0.5,1.2) node[dot] {};}
\DeclareSymbol{3}{0}{\draw (0,0) -- (0,1.2) node[dot] {}; \draw (-.7,1) node[dot] {} -- (0,0) -- (.7,1) node[dot] {};}
\DeclareSymbol{30}{-3}{\draw (0,0) -- (0,-1); \draw (0,0) -- (0,1.2) node[dot] {}; \draw (-.7,1) node[dot] {} -- (0,0) -- (.7,1) node[dot] {};}
\DeclareSymbol{31}{-3}{\draw (0,0) -- (0,-1) -- (1,0) node[dot] {}; \draw (0,0) -- (0,1.2) node[dot] {}; \draw (-.7,1) node[dot] {} -- (0,0) -- (.7,1) node[dot] {};}
\DeclareSymbol{32}{-3}{\draw (0,0) -- (0,-1) -- (1,0) node[dot] {}; \draw (0,0) -- (0,-1) -- (-1,0) node[dot] {}; \draw (0,0) -- (0,1.2) node[dot] {}; \draw (-.7,1) node[dot] {} -- (0,0) -- (.7,1) node[dot] {};}
\DeclareSymbol{20}{-3}{\draw (0,0) -- (0,-1);\draw (-.7,1) node[dot] {} -- (0,0) -- (.7,1) node[dot] {};}
\DeclareSymbol{22}{-3}{\draw (0,0.3) -- (0,-1) -- (1,0) node[dot] {}; \draw (0,0.3) -- (0,-1) -- (-1,0) node[dot] {};\draw (-.7,1) node[dot] {} -- (0,0.3) -- (.7,1) node[dot] {};}
\DeclareSymbol{31p}{-3}{\draw (0,0) -- (0,-1) -- (1,0) node[dot] {}; \draw (0,0) -- (0,1.2) node[dot] {}; \draw (-.7,1) node[dot] {} -- (0,0) -- (.7,1) node[dot] {};
\draw (0,-1) node{\scaleobj{0.5}{\pez}};
\draw (0,-1) node{\scaleobj{0.5}{\pe}}; }
\DeclareSymbol{32p}{-3}{\draw (0,0) -- (0,-1) -- (1,0) node[dot] {}; \draw (0,0) -- (0,-1) -- (-1,0) node[dot] {}; \draw (0,0) -- (0,1.2) node[dot] {}; \draw (-.7,1) node[dot] {} -- (0,0) -- (.7,1) node[dot] {};
\draw (0,-1) node{\scaleobj{0.5}{\pez}};
\draw (0,-1) node{\scaleobj{0.5}{\pe}};}
\DeclareSymbol{22p}{-3}{\draw (0,0.3) -- (0,-1) -- (1,0) node[dot] {}; \draw (0,0.3) -- (0,-1) -- (-1,0) node[dot] {};\draw (-.7,1) node[dot] {} -- (0,0.3) -- (.7,1) node[dot] {};
\draw (0,-1) node{\scaleobj{0.5}{\pez}};
\draw (0,-1) node{\scaleobj{0.5}{\pe}};}
\DeclareSymbol{21p}{-3}{\draw (0,0.3) -- (0,-1) -- (1,0) node[dot] {};
\draw (-.7,1) node[dot] {} -- (0,0.3) -- (.7,1) node[dot] {};
\draw (0,-1) node{\scaleobj{0.5}{\pez}};
\draw (0,-1) node{\scaleobj{0.5}{\pe}};}

\DeclareSymbol{21}{-3}{\draw (0,0.3) -- (0,-1) -- (1,0) node[dot] {};
\draw (-.7,1) node[dot] {} -- (0,0.3) -- (.7,1) node[dot] {};
}

\DeclareSymbol{11p}{-3}{\draw (0,0) node[dot2] {} -- (0,-1) -- (1,0) node[dot] {};
\draw (0,0) -- (0,1.2) node[dot] {};  
\draw (0,-1) node{\scaleobj{0.5}{\pez}};
\draw (0,-1) node{\scaleobj{0.5}{\pe}}; }
\DeclareSymbol{100}{-3}
{\draw[white] (-.4,0) -- (.4,0);
\draw (0,0) node[dot2] {} -- (0,-1)  ;
\draw (0,0) -- (0,1.2) node[dot] {};}  
\usetikzlibrary{mindmap,backgrounds,trees,arrows,decorations.pathmorphing,decorations.pathreplacing,positioning,shapes.geometric,shapes.misc,decorations.markings,decorations.fractals,calc,patterns}

\tikzset{>=stealth',
         cvertex/.style={circle,draw=black,inner sep=1pt,outer sep=3pt},
         vertex/.style={circle,fill=black,inner sep=1pt,outer sep=3pt},
         star/.style={circle,fill=yellow,inner sep=0.75pt,outer sep=0.75pt},
         tvertex/.style={inner sep=1pt,font=\scriptsize},
         gap/.style={inner sep=0.5pt,fill=white}}
\tikzstyle{mybox} = [draw=black, fill=blue!10, very thick,
    rectangle, rounded corners, inner sep=10pt, inner ysep=20pt]
\tikzstyle{boxtitle} =[fill=blue!50, text=white,rectangle,rounded corners]

\tikzstyle{decision} = [diamond, draw, fill=blue!20,
    text width=4.5em, text badly centered, node distance=2.5cm, inner sep=0pt]
\tikzstyle{block} = [rectangle, draw, fill=blue!20,
    text width=2.7cm, text centered, rounded corners, minimum height=3em]

\tikzstyle{line} = [draw, very thick, color=black!50, -latex']

\tikzstyle{cloud} = [draw, ellipse,fill=red!40,
node distance=2.5cm,
    minimum height=1em]

\tikzstyle{cloud2} = [draw, ellipse,fill=red!30, text=white,text width=10em, node distance=2.5cm, text centered, minimum height=4em]

\tikzstyle{cloud3} = [draw, ellipse, fill=cyan!30,
node distance=2.5cm,
    minimum height=3em,
    text width=1.7cm]

\tikzstyle{cloud4} = [draw, ellipse,fill=orange!70, node distance=2.5cm,
   text width=2.1cm,
           minimum height=3em]

\tikzstyle{cloud5} = [draw, ellipse,fill=red!20, node distance=2.5cm,
    text centered,
    text width=3.0cm,
    minimum height=3em]

\tikzstyle{cloud6} = [draw, ellipse,fill=red!20, node distance=2.5cm,
    text width=1.5cm,
    text centered,
    minimum height=3em]

\tikzset{
    position/.style args={#1:#2 from #3}{
        at=(#3.#1), anchor=#1+180, shift=(#1:#2)
    }
}

\newtheorem{theorem}{Theorem} [section]

\newtheorem{lemma}[theorem]{Lemma}
\newtheorem{proposition}[theorem]{Proposition}

\newcommand{\noi}{\noindent}
\newcommand{\Z}{\mathbb{Z}}
\newcommand{\R}{\mathbb{R}}
\newcommand{\T}{\mathbb{T}}
\newcommand{\bul}{\bullet}

\def\norm#1{\|#1\|}

\newcommand{\al}{\alpha}
\newcommand{\be}{\beta}
\newcommand{\dl}{\delta}

\newcommand{\Dl}{\Delta}
\newcommand{\eps}{\varepsilon}

\newcommand{\G}{\Gamma}
\newcommand{\ld}{\lambda}
\newcommand{\Ld}{\Lambda}
\newcommand{\s}{\sigma}

\newcommand{\dd}{\partial}

\renewcommand{\l}{\ell}

\renewcommand{\O}{\Omega}

\renewcommand{\S}{\mathcal{S}}

\newcommand{\N}{\mathbb{N}}

\newtheorem*{ackno}{Acknowledgements}

\newcommand{\C}{\mathcal{C}}

\numberwithin{equation}{section}
\numberwithin{theorem}{section}

\newcommand{\too}{\longrightarrow}

\usepackage{comment}

\newcommand{\Id}{\mathrm{Id}}

\makeatletter
\DeclareRobustCommand\widecheck[1]{{\mathpalette\@widecheck{#1}}}
\def\@widecheck#1#2{%
   \setbox\z@\hbox{\m@th$#1#2$}%
   \setbox\tw@\hbox{\m@th$#1%
      \widehat{%
         \vrule\@width\z@\@height\ht\z@
         \vrule\@height\z@\@width\wd\z@}$}%
   \dp\tw@-\ht\z@
   \@tempdima\ht\z@ \advance\@tempdima2\ht\tw@ \divide\@tempdima\thr@@
   \setbox\tw@\hbox{%
      \raise\@tempdima\hbox{\scalebox{1}[-1]{\lower\@tempdima\box\tw@}}}%
   {\ooalign{\box\tw@ \cr \box\z@}}}
\makeatother

\makeatletter
\@namedef{subjclassname@2020}{%
  \textup{2020} Mathematics Subject Classification}
\makeatother

\begin{document}
\baselineskip = 14pt

\title[Multilinear transference on product spaces]
{Multilinear  transference on product spaces and applications}

\author[\'A. B\'enyi and  T. Oh]{\'Arp\'ad B\'enyi
and Tadahiro Oh}

\address{\'Arp\'ad B\'enyi, Department of Mathematics,
516 High St, Western Washington University, Bellingham, WA 98225,
USA.} \email{benyia@wwu.edu}

%

\address{
Tadahiro Oh, School of Mathematics\\
The University of Edinburgh\\
and The Maxwell Institute for the Mathematical Sciences\\
James Clerk Maxwell Building\\
The King's Buildings\\
Peter Guthrie Tait Road\\
Edinburgh\\
EH9 3FD\\
 United Kingdom,
 and
 School of Mathematics and Statistics, Beijing Institute of Technology,
Beijing 100081, China
}

\email{hiro.oh@ed.ac.uk}



\begin{abstract}
We establish transference results for multilinear Fourier multiplier
operators from Euclidean spaces to product spaces of the form
$\mathbb R^{d_1}\times\mathbb T^{d_2}$. We prove strong and weak-type transference
principles that preserve the corresponding Euclidean operator bounds.
As applications, we obtain bounds for certain multilinear operators, as well as for some multiplier operators arising naturally in PDEs.
\end{abstract}


\maketitle

\tableofcontents

\section{Introduction}\label{SEC:intro}

Transference of boundedness for operators, in both the linear and multilinear settings, is by now a well-established area of research. It originated with the pioneering work of de Leeuw~\cite{deLeeuw} and was subsequently placed in a more general framework by Coifman and Weiss~\cite{CW76}.
One of the classical applications of linear transference
is to deduce the $L^p$-boundedness of the periodic Hilbert transform from that of the Hilbert transform on the real line, which  in turn yields the convergence of Fourier series
for $L^p$-functions on the torus for $1<p<\infty$.
It is therefore natural to seek analogous results for multilinear operators, leading to the development of a corresponding multilinear transference theory. Without attempting to provide an exhaustive bibliography, we mention the contributions of Murray~\cite{Murray}, Grafakos and Weiss \cite{GrafakosWeiss}, Grafakos and Honz\'ik \cite{GrafakosHonzik}, Fan and Sato \cite{FS}, and Blasco \cite{Blasco}. For bilinear transference results on Lorentz and Orlicz spaces, we refer readers to \cite{BlascoVillarroya, BlascoUster}, the latter of which contains an extensive and up-to-date bibliography on the subject. More general multilinear transference results in non-commutative settings can be found, for example, in \cite{CJKM} and \cite{CKV}.

In \cite[Theorem~3]{FS}, the authors proved a multilinear
``Euclidean-to-torus'' transference result for Fourier multiplier operators,
under the assumption that the relevant lattice tuples are Lebesgue points
of the symbol. In certain PDE applications, however, one encounters
product spaces, in which some variables are Euclidean and others
are periodic. In particular, estimates on
$\R\times\T^2$ related to cone multipliers and hyperbolic Riesz potentials
lead naturally to the question of transferring multiplier bounds from
$\R\times\R^2$ to $\R\times\T^2$; see \cite{Zine,FZ26}.
For example, in the analysis of the hyperbolic sinh--Gordon
model \cite{FZ26}, the relevant function spaces are anisotropic
Sobolev spaces \cite[Definition~2.2]{FZ26} and a key
multiplier estimate is established in \cite[Lemma~2.5]{FZ26}. In its proof,
the authors point out the need for precisely such a mixed-type transference
principle: ``Strictly speaking, we need a transference result from
multipliers on $\R\times\R^2$ to multipliers on
$\R\times\T^2$. We were not able to find a ready-made result in this
mixed-type context in the literature''; see
\cite[Footnote~10, p.~21]{FZ26}.

One of the modest goals of this paper is to provide a multilinear
transference principle
from Euclidean spaces to product spaces
for bounded
Fourier multipliers that can serve as a convenient reference for such
applications. In Section~\ref{SEC:Applications}, we illustrate the scope
of our results through several examples, including multilinear
Riesz-type transforms, tensor-like multipliers, and multipliers
arising from PDEs. In the latter setting, we consider the damped
Klein-Gordon propagator as well as a smooth localization to the
characteristic cone of the wave operator, thereby illustrating explicitly
how Euclidean multiplier estimates can be transferred to product spaces.

The paper is organized as follows; Section~\ref{SEC:Notaprelims} contains
the necessary background and preliminary lemmas.
 The main strong and
weak-type transference results are stated and proved in
Sections~\ref{SEC:Multilinear} and \ref{SEC:Weak}, respectively, while,
 in
Section~\ref{SEC:Applications}, we present the applications
described above.

\section{Notation and preliminaries}\label{SEC:Notaprelims}

Given an $m$-tuple $(\xi_1, \dots, \xi_m)$, we
 use the following shorthand notation:
\begin{align}
 \xi_{1\cdots m} = \xi_1 + \cdots +\xi_m.
 \label{sum1}
\end{align}

\noi
Similarly, given an $(m+1)$-tuple $(\al_0, \al_1, \dots, \al_m)$, we set
\begin{align}
 \al_{0\cdots m} = \al_0 + \al_1 + \cdots +\al_m.
 \label{sum2}
\end{align}


Given $d\in\mathbb N$, let $\T^d=\R^d/\Z^d$ denote the $d$-dimensional torus. Given a function $f\in L^1(\R^d)$, we write
\[
\mathcal F(f)(\xi)=\widehat{f}(\xi)=\int_{\R^d}f(x)e^{-2\pi ix\cdot\xi}\,dx, \quad \xi\in\R^d,
\]
for the Fourier transform of $f$. 
With this normalization, the inverse Fourier transform of $f$ is given by
$\mathcal F^{-1}(f)(x)=\widehat f(-x)$.
In this paper, we consider a $d$-dimensional mixed product space
$\R^{d_1}\times\T^{d_2}$ for
$d = d_1+ d_2$, where
$d_1, d_2\in \N$. 
In the following, we write
$x \in \R^{d_1}\times\T^{d_2}$ as $x = (x_1, x_2)$,
where $x_1 \in \R^{d_1}$ and $x_2 \in \T^{d_2}$.
Given a function $f\in L^1(\R^{d_1}\times\T^{d_2})$, we define its Fourier transform by
\[
 \widehat f(\xi,n)
 =\int_{\R^{d_1}}\int_{\T^{d_2}}
 f(x_1, x_2)e^{-2\pi i(x_1\cdot\xi+x_2\cdot n)}\,dx_2\, dx_1,\quad \xi\in \R^{d_1},\,\, n\in\Z^{d_2}.
\]

We let $\S(\R^d)$ denote the Schwartz class of functions on $\R^d$. A result used repeatedly in this work is the Poisson summation formula.
Given
 $f\in\S(\R^d)$,
let $F$ denote the periodization of $f$, given by
\[ F(x)=  \sum_{k\in\Z^d}f(x+k), \quad x \in \R^d.\]

\noi
Then, $F$ is given by the following Fourier series:
%
\begin{equation}
\label{Poisson}
F(x) 
=\sum_{n\in\Z^d}\widehat f(n)e^{2\pi in\cdot x},\quad x\in\R^d,
\end{equation}

\noi
where
 $\widehat f(n)=\int_{\R^{d}}f(x)e^{-2\pi in\cdot x} dx$
denotes the Fourier transform of $f$ (at the frequency $n \in \Z^d$) as a function on $\R^d$.
See, for example, \cite[Theorem 3.2.8]{Gra1}.

Let $d_2 \in \N$.
Given $\eps > 0$,
define $L_\eps$ by
\begin{align}
 L_\eps (x_2)=e^{-\eps\pi |x_2|^2}, \quad
x_2\in\R^{d_2}.
\label{L1}
\end{align}

\noi
The following lemma is well known and will be used in the proof of our main transference result; see, for example, \cite[(4.3.13)]{Gra1}. For readers' convenience,  we sketch its proof.

\begin{lemma}
\label{lem:gaussian-average}
Let $f\in L^1(\T^{d_2})$ and extend it $\Z^{d_2}$-periodically onto $\R^{d_2}$.  Then,
we have
\begin{equation}
 \lim_{\eps\to 0}
 \eps^{d_2/2}\int_{\R^{d_2}}f(x_2)L_\eps (x_2)\, dx_2
 =\int_{\T^{d_2}}f(x_2)\, dx_2.
 \label{eq:gaussian-average}
\end{equation}

\noi
Moreover, given $0<r<\infty$ and $F\in L^r(\R^{d_1}\times\T^{d_2})$, we have
\begin{equation}
 \lim_{\eps\to 0}
 \eps^{d_2/(2r)}
 \norm{F(x_1, x_2)L_{\eps/r}(x_2)}_{L^r(\R^{d_1}\times\R^{d_2})}
 =\norm{F}_{L^r(\R^{d_1}\times\T^{d_2})},
 \label{eq:norm-average}
\end{equation}

\noi
where $F$ on the left-hand side
is the periodic extension
\textup{(}in the second variable $x_2$\textup{)}
of $F$ onto $\R^{d_1}\times\R^{d_2}$

\end{lemma}

\begin{proof}
Using the periodicity of $f$ and Fubini's  theorem, we get
\[
\eps^{d_2/2}\int_{\R^{d_2}}f(x_2)L_\eps(x_2)\, dx_2
 =\int_{[0, 1)^{d_2}} f(x_2)P_\eps(x_2)\, dx_2,
\]
where $P_\eps$ is given by
$$
P_\eps (x_2)=\eps^{d_2/2}\sum_{k\in\Z^{d_2}}L_\eps (x_2+k).
$$
The Poisson summation formula \eqref{Poisson}
with the invariance of the Gaussian under the Fourier transform then yields
\[
P_\eps(x_2)=\sum_{n\in\Z^{d_2}}L_{1/\eps}(n)e^{2\pi i n\cdot x_2}.
\]
Hence, as $\eps\to 0$, we have
\[
 \sup_{x_2\in [0, 1)^{d_2}}\left|P_\eps(x_2)-1\right|
 \leq\sum_{n\in\Z^{d_2}\setminus\{0\}} L_{1/\eps}(n)\too 0,
 \]

\noi
where the last step follows from \eqref{L1}.
This
 proves \eqref{eq:gaussian-average}.
 As for \eqref{eq:norm-average},
 given $F \in L^r( \R^{d_1}\times\T^{d_2})$,
 simply define $f \in L^1( \T^{d_2})$ by
\[
 f(x_2)=\int_{\R^{d_1}}|F(x_1, x_2)|^r\, dx_1,\quad  x_2\in\T^{d_2},
\]

\noi
and apply
\eqref{eq:gaussian-average} to $f$. \end{proof}

In Section~\ref{SEC:Multilinear}, we find it useful to work with normalized versions of the Gaussian density $L_\eps$ introduced in \eqref{L1}.
 Namely, given $\alpha, \eps>0$,
 set
 \begin{align}
G_{\eps, \alpha}=(\eps\alpha)^{-d_2/2}L_{1/(\epsilon\alpha)}.
 \label{G1}
 \end{align}

\noi
It is straightforward
to see that $\int_{\R^{d_2}} G_{\eps, \alpha}(x_2)\, dx_2=1$ and
\begin{equation}
G_{\eps, \alpha_1}*G_{\eps, \alpha_2}=G_{\eps, \alpha_1+\alpha_2}.
\label{G1a}
\end{equation}
For $m\in\mathbb N$, we identify $\R^{md_2}=\underbrace{\R^{d_2}\times\cdots\times\R^{d_2}}_{m\text{-times}}$, and write
\(
\boldsymbol y=(y_1,\ldots, y_m)\in\R^{md_2}.
\) A repeated application of the convolution identity \eqref{G1a} immediately yields the following result.

\begin{lemma}
\label{lem:resonance}
Let $m\in\mathbb N$ and $\alpha_0, \al_1,\ldots, \al_m>0$ such that
\(
\displaystyle\sum_{j=0}^m \al_j=1.
\)
For $x_2\in\R^{d_2}$ and $\boldsymbol y\in\R^{md_2}$, define
\begin{align}
 \Gamma_{\eps, x_2}(\boldsymbol y)
 =\eps^{d_2/2}G_{\eps, \al_0}
(x_2- y_{1\cdots m})
\prod_{j=1}^m G_{\eps, \al_j}(y_j).
 \label{G2}
\end{align}

\noi
Then, we have
\begin{equation}
 \int_{\R^{md_2}}\Gamma_{\eps, x_2}(\boldsymbol y)\, d\boldsymbol y
 =G_{\eps,  \al_{0\cdots m}}(x_2)=L_{1/\eps}(x_2).
\label{G2a}
\end{equation}

\noi
In \eqref{G2} and \eqref{G2a}, we used the shorthand notation \eqref{sum1}
and \eqref{sum2}.

\end{lemma}

Note that, in particular, if $x_2\neq 0$, then we have
\begin{align*}
\lim_{\eps\to 0}\int_{\R^{md_2}}\Gamma_{\eps, x_2}(\boldsymbol y)\, d\boldsymbol y=0,
\end{align*}

\noi
while, if $x_2=0$,  then we have
\begin{equation}
 \lim_{\eps\to 0}
 \int_{\R^{md_2}} H(\boldsymbol y)\Gamma_{\eps, 0}(\boldsymbol y)\,d\boldsymbol y =H(0)
 \label{G4}
\end{equation}
for any bounded function  $H$ that is continuous at the origin.

Finally, we recall the standard definition of  a multilinear Fourier multiplier operator.
Fix
 $m\in\mathbb N$.
  As above, we use the notation
\(
 \boldsymbol\xi=(\xi_1,\ldots,\xi_m)\in \R^{md}.
\)
Given
a  continuous and bounded function
$\sigma:\R^{md}\to\mathbb C$, we say that $T_\sigma$ is an \emph{$m$-linear Fourier multiplier operator} with symbol $\sigma$,  if we have
\begin{align*}
T_\sigma(f_1,\ldots, f_m)(x)=\int_{\R^{md}}\sigma(\boldsymbol\xi)\prod_{j=1}^m\widehat{f_j}(\xi_j)e^{2\pi ix\cdot \sum_{j=1}^m \xi_j}\,d\boldsymbol\xi,
\quad x\in\R^d,
\end{align*}

\noi
for  any  $f_j\in\S(\R^d)$, $j = 1, \dots, m$.
Letting $K_{T_\sigma}:=\mathcal F^{-1}_{\R^{md}}(\sigma)$, we see that
\begin{align*}
T_{\sigma}(f_1,\ldots, f_m)(x)
& =\int_{\R^{md}}K_{T_\sigma}(\boldsymbol y)\prod_{j=1}^m f_j(x-y_j)\,dy_1\ldots\,dy_m.
\end{align*}

%

\noi
It is an immediate  consequence of Minkowski's integral inequality and H\"older's inequality that if $K_{T_\sigma}\in L^1(\R^{md})$, then $T_\sigma$ is bounded from $L^{p_1}(\R^d)\times\cdots\times L^{p_m}(\R^d)\to L^p(\R^d)$ for any $1\leq p, p_j\le\infty$, $j = 1, \dots, m$,  such that $\frac{1}{p}=\sum_{j=1}^m \frac{1}{p_j}$; indeed, we have the following bound:
\begin{align}
\|T_\sigma\|_{L^{p_1}\times\cdots\times L^{p_m}\to L^p}\leq\|K_{T_\sigma}\|_{L^1(\R^{md})}.
\label{T0}
\end{align}

\section{Multilinear transference: strong-type results}\label{SEC:Multilinear}

Let $m\in\mathbb N$ and $d, d_1, d_2\in\mathbb N$ such that $d=d_1+d_2$.
Let $\sigma: \R^{md_1}\times\R^{md_2}\to\mathbb C$ be bounded and continuous. As before, we identify $\R^{md}=\R^{md_1}\times\R^{md_2}$ and write
\[
\boldsymbol\xi=(\xi_1,\ldots,\xi_m)\in\R^{md_1}
\quad \text{and}\quad
 \boldsymbol\eta=(\eta_1,\ldots,\eta_m)\in\R^{md_2}.
\]

\noi
Then, the
$m$-linear Fourier multiplier operator $T_\s$  with symbol $\sigma$ is given by
\begin{equation}
\label{T1}
\begin{split}
& T_\sigma (f_1,\ldots, f_m)(x_1, x_2)\\
& \quad =\int_{\R^{md_1}}\int_{\R^{md_2}}\sigma(\boldsymbol\xi, \boldsymbol\eta)\prod_{j=1}^m\widehat{f_j}(\xi_j, \eta_j)e^{2\pi i\sum_{j=1}^m (x_1\cdot\xi_{j}+x_2\cdot\eta_j)}\,d\boldsymbol\xi\,d\boldsymbol\eta
\end{split}
\end{equation}

\noi
for
 $x_{1}\in\R^{d_1}$,  $x_{2}\in\R^{d_2}$,
 and
 $f_j\in\S(\R^{d})$, $j = 1, \dots, m$,

Next, we introduce a multilinear Fourier multiplier operator
on the product space $\R^{d_1}\times\T^{d_2}$.
As before, we write
\[ \boldsymbol\xi = (\xi_1, \dots, \xi_m) \in\R^{md_1}
\quad \text{and}\quad
\boldsymbol n = (n_1, \dots, n_m)\in\Z^{md_2}.\]

\noi
Given a bounded function
 $\sigma:\R^{md_1}\times\Z^{md_2}\to\mathbb C$,
 we define $S_\s$
 by setting
\begin{equation}
\label{T2}
\begin{split}
& S_\sigma (F_1,\ldots, F_m)(x_1, x_2)\\
& \quad
=\sum_{\boldsymbol n\in\Z^{md_2}}\int_{\R^{md_1}}\sigma(\boldsymbol\xi, \boldsymbol n)\prod_{j=1}^m\widehat{F_j}(\xi_{j}, n_{j})e^{2\pi i\sum_{j=1}^m (x_1\cdot\xi_{j}+x_2\cdot n_{j})}\,d\boldsymbol\xi
\end{split}
\end{equation}

\noi
for
$x_{1}\in\R^{d_1}$,
$x_2\in\T^{d_2}$,
 and
 $F_j\in\S(\R^{d_1}\times\T^{d_2})$, $j = 1, \dots, m$.
 Here, the notation
$ \S(\R^{d_1}\times\T^{d_2})$
denotes the collection of functions
$F \in C^\infty (\R^{d_1}\times\T^{d_2})$
such that
\[
\sup_{(x_1, x_2) \in \R^{d_1} \times \T^{d_2}}
|x_1^{\al_1}
\dd_{x_1}^{\be_1} \dd_{x_2}^{\be_2} F(x_1, x_2)|< \infty\]

\noi
for any multi-indices $\al_1 , \be_1 \in \N_0^{md_1}$ and
$\be_2  \in \N_0^{md_2}$, where $\N_0 = \N \cup\{0\}$.

In the following,
we use $K_{T_\s}$ (and $K_{S_\s}$)
to
 denote the convolution kernel of $T_\sigma$
on $\R^{md_1}\times\R^{md_2}$
defined  in \eqref{T1} (and $S_\sigma$
on $\R^{md_1}\times\T^{md_2}$
defined in  \eqref{T2}, respectively).
Namely, we have
 \[
K_{T_\s}=\mathcal F^{-1}_{\R^{md}}(\sigma)\qquad\text{and}\qquad
K_{S_\s}=\mathcal F^{-1}_{\R^{md_1}\times\T^{md_2}}(\sigma).
\]

\noi
With the notation above, we first go over the following simple  transference result.

\begin{proposition}\label{LEM:multitrans1}
Let $m\in\mathbb N$ and $d, d_1,d_2\in\N$ with $d=d_1+d_2$.
Given a bounded and continuous function $\sigma: \R^{md_1}\times\R^{md_2}\to\mathbb C$,
let $T_\s$ and $S_\s$ be as in \eqref{T1} and \eqref{T2}, respectively.
If $K_{T_\s} \in L^1(\R^{md_1}\times\R^{md_2})$, then $K_{S_\s}\in
L^1(\R^{md_1}\times\T^{md_2})$
with the  bound\textup{:}
\begin{align}
\|K_{S_\s}\|_{L^1(\R^{md_1}\times\T^{md_2})}
\le
\|K_{T_\s}\|_{L^1(\R^{md})} < \infty.
\label{multitrans1}
\end{align}

\noi
In particular,
 given any  \(1\leq p, p_1,\ldots, p_m<\infty\)
satisfying
\begin{align}
\displaystyle\frac{1}{p}=\displaystyle\sum_{j=1}^m \displaystyle\frac{1}{p_j},
\label{T3}
\end{align}
the operator  $S_\s$ is bounded from $L^{p_1}(\R^{d_1}\times\T^{d_2})\times\cdots \times L^{p_m}(\R^{d_1}\times\T^{d_2})$ to $L^p(\R^{d_1}\times\T^{d_2})$
with the following operator bound\textup{:}
\begin{align}
\| S_\s\|_{L^{p_1}\times\cdots\times L^{p_m}\to L^{p}}
\le
\|K_{T_\s}\|_{L^1(\R^{md})}.
\label{multitrans2}
\end{align}
\end{proposition}

\begin{proof}
Let $K_{T_\sigma}\in L^1(\mathbb R^{md_1}\times
\mathbb R^{md_2})$ be the convolution kernel of $T_\sigma$. Define its
periodization in the second variable by
\begin{align*}
    \widetilde K(\boldsymbol x,\boldsymbol y)
    = \sum_{\boldsymbol k\in\mathbb Z^{md_2}}K_{T_\sigma}(\boldsymbol x,\boldsymbol y+\boldsymbol k),
    \qquad
    (\boldsymbol x,\boldsymbol y)\in\mathbb R^{md_1}\times\mathbb T^{md_2}.
\end{align*}

\noi
By the Fubini--Tonelli theorem, the sum converges absolutely for almost
every $(\boldsymbol x,\boldsymbol  y)$, and $\widetilde K\in
L^1(\mathbb R^{md_1}\times\mathbb T^{md_2})$. Moreover,
we have
\begin{align}
\begin{split}
    \|\widetilde K\|_{L^1(\mathbb R^{md_1}\times\mathbb T^{md_2})}
    &\le
    \int_{\mathbb R^{md_1}}\int_{\mathbb T^{md_2}}
    \sum_{k\in\mathbb Z^{md_2}}
    |K_{T_\sigma}(\boldsymbol x,\boldsymbol y+\boldsymbol k)|\,d\boldsymbol yd\boldsymbol x \\
    &=
    \int_{\mathbb R^{md_1}}\int_{\mathbb R^{md_2}}
    |K_{T_\sigma}(\boldsymbol x,\boldsymbol y)|\,d\boldsymbol yd\boldsymbol x \\
    &=
    \|K_{T_\sigma}\|_{L^1(\mathbb R^{md})}.
\end{split}
\label{Z1}
\end{align}

We claim that $\widetilde K=K_{S_\sigma}$. Indeed, for every
$n\in\mathbb Z^{md_2}$, the $\boldsymbol n$th Fourier coefficient of
$\widetilde K$ in the periodic variable is given, for almost every
$\boldsymbol  x\in\mathbb R^{md_1}$, by
\begin{align*}
    \int_{\mathbb T^{md_2}}\widetilde K(\boldsymbol x,\boldsymbol y)e^{-2\pi i \boldsymbol n\cdot \boldsymbol y}\,d\boldsymbol y
    &=
    \int_{\mathbb R^{md_2}}
    K_{T_\sigma}(\boldsymbol x,\boldsymbol y)e^{-2\pi i \boldsymbol n\cdot \boldsymbol y}\,d\boldsymbol y\\
    &=
    \mathcal F_{\mathbb R^{md_2}}^{(2)}
    (K_{T_\sigma})(\boldsymbol x,\boldsymbol n).
\end{align*}
Taking the Fourier transform also in
$\boldsymbol x\in\mathbb R^{md_1}$, we obtain
\begin{align*}
    \mathcal F_{\mathbb R^{md_1}}
    \left[
    \int_{\mathbb T^{md_2}}
    \widetilde K(\,\cdot\,,\boldsymbol y)e^{-2\pi i \boldsymbol n\cdot \boldsymbol y}\,d\boldsymbol y
    \right](\boldsymbol \xi)
    =
    \sigma(\boldsymbol \xi,\boldsymbol n).
\end{align*}

\noi
Thus,  $\widetilde K$ has Fourier transform $\sigma(\boldsymbol \xi,\boldsymbol n)$ on
$\mathbb R^{md_1}\times\mathbb Z^{md_2}$ and
hence
$\widetilde K=K_{S_\sigma}$.
Therefore,
\eqref{multitrans1} follows from \eqref{Z1}.
The bound \eqref{multitrans2} then follows from \eqref{multitrans1} and \eqref{T0}.
\end{proof}

It is worth noting that the conclusion of Proposition~\ref{LEM:multitrans1} is false under the assumption $0<p<1$: Minkowski's inequality is not available in this case and, in fact, it is well known that convolution by an \(L^1\)-kernel need not be bounded
on \(L^p\) when \(p<1\).

The next proposition establishes a general transference result on product spaces in the non-endpoint case. Unlike the large-box cutoff argument used in \cite{FS},
we proceed with de Leeuw's argument via Gaussians as in the linear proof presented in \cite[Theorem~4.3.7 and Lemmas~4.3.8 and 4.3.9]{Gra2}.

\begin{proposition}
\label{LEM:multitrans2}
Let $m\in\mathbb N$ and $d, d_1,d_2\in\N$ with $d=d_1+d_2$.
Given a bounded and continuous function $\sigma: \R^{md_1}\times\R^{md_2}\to\mathbb C$,
let $T_\s$ and $S_\s$ be as in \eqref{T1} and \eqref{T2}, respectively.
Suppose that $T_\s$
is bounded from
\(
 L^{p_1}(\R^d)\times\cdots\times L^{p_m}(\R^d)\) to $L^p(\R^d)$
for some
\(
 1< p, p_1,\ldots,p_m<\infty\), satisfying \eqref{T3}.
Then,  $S_\sigma$ in \eqref{T2} is bounded from
\(
 L^{p_1}(\R^{d_1}\times\T^{d_2})\times\cdots\times
 L^{p_m}(\R^{d_1}\times\T^{d_2})\) to \(
 L^p(\R^{d_1}\times\T^{d_2})\),
satisfying the bound\textup{:}
\begin{equation}
 \norm{S_\sigma}_{L^{p_1}\times\cdots\times L^{p_m}\to L^p}
 \leq
 \norm{T_\sigma}_{L^{p_1}\times\cdots\times L^{p_m}\to L^p}.
 \label{eq:mainbound}
\end{equation}
\end{proposition}

Before presenting the proof of Proposition~\ref{LEM:multitrans2}, we remark that the constant 1 in the bound~\eqref{eq:mainbound} is sharp. This is easily seen by considering the constant symbol $\sigma=1$, in which case both multilinear multiplier operators reduce to  pointwise multiplication, namely
\[
T_1(f_1,\ldots,f_m)=\prod_{j=1}^m f_j,
 \qquad
 S_1(F_1,\ldots,F_m)=\prod_{j=1}^m F_j.
\]
Using H\"older's inequality, we see that both operator norms in
\eqref{eq:mainbound} are exactly 1.

As the proof below shows, the hypothesis on the symbol $\sigma$ can be relaxed in the spirit of \cite[Theorem 3]{FS} to only require $\sigma\in L^\infty(\R^{md})$ and that,
for every
 $\boldsymbol n\in\Z^{md_2}$ and a.e. $\boldsymbol\xi\in\R^{md_1}$, the origin is a Lebesgue point of
\(\boldsymbol y\mapsto \sigma (\boldsymbol\xi, \boldsymbol y+\boldsymbol n).
\) Moreover,
the restriction $p > 1$
comes from the use of Lemma \ref{lem:resonance},
which requires $\al_0$ in \eqref{T4a} to be positive.

\begin{proof} [Proof of Proposition~\ref{LEM:multitrans2}]
 Let
\begin{align}
A=\norm{T_\sigma}_{L^{p_1}(\R^d)\times\cdots\times L^{p_m}(\R^d)\to L^p(\R^d)}, \quad \alpha_0=\frac{1}{p'},\quad \text{and}\quad \alpha_j=\frac{1}{p_j}
\label{T4a}
\end{align}

\noi
for $j = 1, \dots, m$.
Then, \eqref{T3} reduces to
\begin{equation}
 \alpha_0+\alpha_1+\cdots+\alpha_m=1.
 \label{T4}
\end{equation}
Next, we define a dense class of elementary functions on $\mathbb R^{d_1}\times\T^{d_2}$ as follows;
 we say that $F$ is an elementary function (defined on $\mathbb R^{d_1}\times\T^{d_2}$) if there exists
a finite set $\Ld \subset \Z^{d_2}$  and a collection  $\{f_n\}_{n\in\Ld } \subset \S(\R^{d_1})$ such that
\begin{equation*}
 F(x_1, x_2)=\sum_{n\in \Ld }f_n(x_1)e^{2\pi i n\cdot x_2},
 \quad (x_1, x_2) \in\R^{d_1}\times \T^{d_2}.
\end{equation*}
We denote by $\mathcal E$ the collection of these elementary functions, namely
\begin{align}
\mathcal E
=\bigg\{F=\sum_{n\in\Lambda}f_n\otimes e_n^{(2)},\,\, \Lambda\subset\Z^{d_2}\,\,\text{finite},\,\,f_n\in\mathcal S(\R^{d_1})\bigg\},
\label{T5a}
\end{align}

\noi
where we wrote $e_n^{(2)}(x_2)=e^{2\pi in\cdot x_2},\, x_2\in\T^{d_2}.$  By the density of Schwartz functions and trigonometric polynomials, it follows that $\mathcal E$ is dense in $L^r(\R^{d_1}\times\T^{d_2})$, $0<r<\infty$.
In the following,  we often view a function on $\R^{d_1}\times\T^{d_2}$ as a function on $\R^{d_1}\times\R^{d_2}$ which is periodic in $x_2\in\R^{d_2}$.

For $j = 0, 1, \dots, m$,
let $F_j\in\mathcal E$.
Namely, there exists a finite set $\Ld_j \subset \Z^{d_2}$
 and $f_{j, n_j}\in\S(\R^{d_1})$ such that
\begin{equation}
 F_j=\sum_{n_j\in\Lambda_j}f_{j,n_j}\otimes e_{n_j}^{(2)}.
 \label{T6}
\end{equation}

\noi
Then,
for $\eps>0$, define the Schwartz functions $F_{j, \eps}$ on $\R^{d}$ by
\begin{align}
 F_{j,\eps}(x_1, x_2)=F_j(x_1, x_2)L_{\eps\alpha_j}(x_2),
\label{T6a}
\end{align}

\noi
where $L_{\eps\alpha_j}$ is as in \eqref{L1}.
Note that by Lemma~\ref{lem:gaussian-average}, we have
\begin{align}
\begin{split}
\lim_{\eps\to 0}
 \eps^{d_2/(2p')}
 \norm{F_{0, \eps}}_{L^{p'}(\R^{d})}
 &=\norm{F_0}_{L^{p'}(\R^{d_1}\times\T^{d_2})}, \\
\lim_{\eps\to 0}
 \eps^{d_2/(2p_j)}
 \norm{F_{j, \eps}}_{L^{p_j}(\R^{d})}
 &=\norm{F_j}_{L^{p_j}(\R^{d_1}\times\T^{d_2})}, \quad j = 1, \dots, m.
\end{split}
 \label{T7}
\end{align}

\medskip

\noindent
$\bul$~\textbf{Step 1:}
Given $\eps > 0$, let
\begin{align}
I_\eps=\eps^{d_2/2}\int_{\R^d}
 T_\sigma(F_{1,\eps},\ldots,F_{m,\eps})(x)
 \overline{F_{0, \eps}(x)}\, dx
\label{T7a}
\end{align}

\noi
for $F_{j, \eps}$, $j = 0, 1, \dots, m$,  as in \eqref{T6a}.
We claim that
\begin{equation}
 \lim_{\eps\to 0}I_\eps
 =\int_{\R^{d_1}\times\T^{d_2}}
 S_\sigma(F_1,\ldots,F_m)(x_1, x_2)\overline{F_0(x_1, x_2)}
 \,dx_1dx_2.
 \label{T8}
\end{equation}

From \eqref{T6a} and \eqref{G1}, we have
\begin{equation}
 \widehat F_{j,\eps}(\xi,\eta)
 =\sum_{n_j\in \Ld_j}\widehat f_{j,n_j}(\xi)G_{\eps, \alpha_j}(\eta-n_j).
\label{T9}
\end{equation}

\noi
For $\boldsymbol\xi \in \R^{md_1}$,   $\boldsymbol n = (n_1, \dots, n_m) \in \Z^{md_2}$,
and  $n_0 \in \Z^{d_2}$,
define  $J_\eps(\boldsymbol\xi,\boldsymbol n, n_0)$ by setting
\begin{equation}
\label{T10}
 J_\eps(\boldsymbol\xi,\boldsymbol n, n_0)=\int_{\R^{md_2}}
 \sigma(\boldsymbol\xi, \boldsymbol n+\boldsymbol y)
 \Gamma_{\eps,\,n_{1\cdots m}-n_0}(\boldsymbol y)
 \, d\boldsymbol y,
\end{equation}

\noi
where $ \Gamma_{\eps,\,n_{1\cdots m}-n_0}$ is as in \eqref{G2}
with the notation \eqref{sum1}.
Noting that
 the sums in \eqref{T6}, defining the elementary functions $F_j$ are finite,
 it follows from  Parseval's identity,  the definition \eqref{T1} of $T_\sigma$,
\eqref{T9},  \eqref{G1a},
 \eqref{T4}  and a change of variables $y_j=\eta_j-n_j$ with \eqref{T10}
 (see also \eqref{G2}) that
\begin{align}
 I_\eps
 &=\sum_{\substack{n_j\in \Ld_j\\ j = 0, 1, \dots, m}}
   \int_{\R^{md_1}}
   J_\eps(\boldsymbol\xi,\boldsymbol n, n_0)
\prod_{j=1}^m\widehat f_{j,n_j}(\xi_j)\cdot
    \overline{\widehat f_{0, n_0} (\xi_{1\cdots m})}
   \,d\boldsymbol\xi.
\label{T11}
\end{align}

By applying
 Lemma~\ref{lem:resonance} to \eqref{T10}, we are led to consider two cases.
In the non-resonant case: $n_0 \ne n_{1\cdots m},$ we have
\begin{align}
|J_\eps(\boldsymbol\xi,\boldsymbol n, n_0)|
 \leq\norm{\sigma}_{L^\infty(\R^{md})}L_{1/\eps}(n_{1\cdots m}- n_0)
 \too 0,
\label{T12}
\end{align}

\noi
as $\eps\to 0$.
In the resonant case: $n_0 = n_{1\cdots m} $,
it follows from the boundedness and continuity of $\sigma$ with \eqref{G4}
that
\begin{align}
\lim_{\eps\to 0}J_\eps(\boldsymbol\xi,\boldsymbol n, n_0)=
\sigma (\boldsymbol\xi, \boldsymbol n).
\label{T13}
\end{align}

\noi
Moreover, in both cases, we have \(|J_\eps|\leq\norm{\sigma}_{L^\infty(\R^{md})}\). For every fixed finite tuple of indices, the remaining function of $\boldsymbol\xi$ in \eqref{T11} is integrable as
\[\int_{\R^{md_1}}
 \prod_{j=1}^m |\widehat f_{j,n_j}(\xi_j)|
\cdot |\widehat f_{0, n_0}(\xi_{1\cdots m})|d\boldsymbol\xi
 \leq
 \norm{\widehat f_{0, n_0}}_{L^\infty(\R^{d_1})}
 \prod_{j=1}^m\norm{\widehat f_{j,n_j}}_{L^1(\R^{d_1})}<\infty.
\]

\noi
By applying the
 Lebesgue dominated convergence
 to $I_\eps$ in  \eqref{T11}
 with \eqref{T12} and \eqref{T13},
 we obtain
\begin{align*}
 \lim_{\eps\to 0}I_\eps =
 \sum_{\substack{n_j\in \Ld_j\\ j = 0, 1, \dots, m\\ n_0 = n_{1\cdots m}}}
    \int_{\R^{md_1}}
\sigma (\boldsymbol\xi, \boldsymbol n).
\prod_{j=1}^m\widehat f_{j,n_j}(\xi_j)\cdot
    \overline{\widehat f_{0, n_0} (\xi_{1\cdots m})}
   \,d\boldsymbol\xi.
\end{align*}

\noi
In view of \eqref{T2} and \eqref{T6}
with Parseval's identity,
this  is exactly
 the Fourier expansion of the right-hand side of~\eqref{T8}, proving our claim \eqref{T8}.

\medskip

\noindent
$\bul$~\textbf{Step 2:} We now prove the bound \eqref{eq:mainbound} via duality.
By applying
H\"older's inequality and  the boundedness of $T_\sigma$
to $I_\eps$ in
\eqref{T7a}
with  \(1/p'+\sum_{j=1}^m 1/p_j=1\), we have
\begin{equation}
\begin{split}
 |I_\eps|
&  \leq A\,\eps^{d_2/2}
 \norm{F_{0, \eps}}_{L^{p'}(\R^d)}
 \prod_{j=1}^m\norm{F_{j,\eps}}_{L^{p_j}(\R^d)}
\\
&  =A
 \left(\eps^{d_2/(2p')}
       \norm{F_{0, \eps}}_{L^{p'}(\R^d)}\right)
 \prod_{j=1}^m
 \left(\eps^{d_2/(2p_j)}
       \norm{F_{j,\eps}}_{L^{p_j}(\R^d)}\right),
\end{split}
 \label{eq:Ieps-bound}
\end{equation}

\noi
where $A$ is as in \eqref{T4a}.
Letting $\eps\to 0$ in \eqref{eq:Ieps-bound} together with \eqref{T8} and \eqref{T7},
we obtain
\begin{equation}
\label{eq:final}
\begin{split}
& \bigg|\int_{\R^{d_1}\times\T^{d_2}}
 S_\sigma(F_1,\ldots,F_m)(x_1, x_2)\overline{F_0(x_1, x_2)}\,dx_1dx_2\bigg|\\
& \hphantom{XXXXXXXXXX}
  \leq A
 \norm{F_0}_{L^{p'}(\R^{d_1}\times\T^{d_2})}
 \prod_{j=1}^m\norm{F_j}_{L^{p_j}(\R^{d_1}\times\T^{d_2})},
\end{split}
\end{equation}
for any elementary functions $F_j\in\mathcal E$,
$j = 0, 1, \dots, m$.
Therefore, the bound \eqref{eq:mainbound} follows from~\eqref{eq:final}
and the density of $\mathcal E$ in $L^{p_j}(\R^{d_1}\times\T^{d_2})$.
\end{proof}

\section{Multilinear transference: weak-type result}\label{SEC:Weak}

In this section, we establish
the following weak-type transference result on product spaces.

\begin{proposition}
\label{LEM:multitrans3}
Let $m\in\mathbb N$ and $d, d_1,d_2\in\N$ with $d=d_1+d_2$.
Given a bounded and continuous function $\sigma: \R^{md_1}\times\R^{md_2}\to\mathbb C$,
let $T_\s$ and $S_\s$ be as in \eqref{T1} and \eqref{T2}, respectively.
Suppose that $T_\s$
is bounded from
\(
 L^{p_1}(\R^d)\times\cdots\times L^{p_m}(\R^d)\) to $L^{p, \infty}(\R^d)$
for some
\(
 0<p, p_1,\ldots,p_m<\infty\), satisfying \eqref{T3}.
Then,  $S_\sigma$ in \eqref{T2} is bounded from
\(
 L^{p_1}(\R^{d_1}\times\T^{d_2})\times\cdots\times
 L^{p_m}(\R^{d_1}\times\T^{d_2})\) to \(
 L^{p, \infty}(\R^{d_1}\times\T^{d_2})\),
satisfying the bound\textup{:}
\begin{equation}
 \norm{S_\sigma}_{L^{p_1}\times\cdots\times L^{p_m}\to L^{p, \infty}}
 \leq
 \norm{T_\sigma}_{L^{p_1}\times\cdots\times L^{p_m}\to L^{p, \infty}}.
 \label{S1}
\end{equation}

\end{proposition}

\begin{proof}
In  the proof of Proposition~\ref{LEM:multitrans2},
 we  used Gaussian localization and the $(L^p, L^{p'})$-duality.
For a weak-type estimate,
 this duality is less natural,
 and thus  we adapt  the cutoff argument from the proof of  \cite[Theorem 4]{FS}.

Let
\begin{align}
B=\norm{T_\sigma}_{L^{p_1}(\R^d)\times\cdots\times L^{p_m}(\R^d)
 \to L^{p,\infty}(\R^d)}.
\label{S2}
\end{align}

\noi
Namely,
$B$ is the smallest constant such that
\begin{equation}
\label{eq:weakTbound}
 \bigl\lvert
 \{x\in\R^d:
 \lvert T_\sigma(f_1,\ldots, f_m)(x)\rvert>\lambda\}
 \bigr\rvert
 \leq
 \left(
 \frac{B\prod_{j=1}^m\norm{f_j}_{L^{p_j}(\R^d)}}{\lambda}
 \right)^p
\end{equation}

\noi
for any \(\lambda>0\).
As in the proof of Proposition~\ref{LEM:multitrans2}, it suffices to prove our estimates for the class $\mathcal E$ of elementary functions
defined in \eqref{T5a}.
In the following, we fix
  $F_j\in\mathcal E$,  $j = 1, \dots, m$,  as in \eqref{T6}.

Let $Q=[-\frac{1}{2}, \frac{1}{2}]^{d_2}$, $Q_K=
 \left[-\frac{1}{2}-\frac{1}{K},\frac{1}{2}+\frac{1}{K}\right]^{d_2}$ for a fixed \(K\in\N\), and choose \(\Psi\in C_c^\infty(\R^{d_2}; [0, 1])\) such that
\begin{equation*}
 \Psi=1 \text{ on } Q
 \qquad \text{and}\qquad
 \operatorname{supp}\Psi\subset Q_K.
\end{equation*}
Following \cite[(3.1)]{FS}, for $N\in\N$ define
the error term $ E_N (F_1,\ldots, F_m)$ by setting
\begin{equation}
\label{eq:EN}
\begin{split}
 E_N (F_1,\ldots, F_m)(x_1, x_2)
 & =\Psi_N(x_2)^m S_\sigma(F_1,\ldots, F_m)(x_1, x_2)\\
 & \quad -T_\sigma(F_{1,N},\ldots, F_{m,N})(x_1, x_2),
\end{split}
\end{equation}
where
\begin{align}
 \Psi_N(x_2)=\Psi(x_2/N)
 \quad
 \text{and}\quad
 F_{j,N}(x_1,x_2)=F_j(x_1,x_2)\Psi_N(x_2).
\label{S4}
\end{align}

\medskip

\noindent
$\bul$~\textbf{Step 1:} We first show  that the following uniform approximation holds:
\begin{equation}
\label{eq:ENuniform}
 \norm{E_N (F_1,\ldots, F_m)}_{L^\infty(\R^d)}\longrightarrow0
 \quad\text{as }N\to\infty.
\end{equation}

From \eqref{S4} with \eqref{T6}, we have
\[
 \widehat F_{j,N}(\xi_j,\eta_j)
 =
 \sum_{n_j\in\Lambda_j}
 \widehat f_{j, n_j}(\xi_j)\,
 N^{d_2}\widehat\Psi\bigl(N(\eta_j-n_j)\bigr)
\]

\noi
Then, by a change of variables  $u_j=N(\eta_j-n_j)$, we obtain
\begin{align}
\begin{split}
&T_\sigma(F_{1,N},\ldots,F_{m,N})(x_1,x_2)
\\
&\quad =
\sum_{\substack{n_j\in\Lambda_j\\  j = 1, \dots, m}}
\int_{\R^{md_1}}\int_{\R^{md_2}}
 \sigma\left(
 \boldsymbol\xi,
 \boldsymbol n+\boldsymbol u/N
 \right)
 \prod_{j=1}^m
 \widehat f_{j,n_j}(\xi_j)\widehat\Psi(u_j)
\\
&\hspace{2.8cm}\times
 e^{2\pi i x_1\cdot\sum_{j=1}^m\xi_j}
 e^{2\pi i x_2\cdot\sum_{j=1}^m n_j}
 e^{2\pi i (x_2/N)\cdot\sum_{j=1}^m u_j}
 \,d\boldsymbol u\,d\boldsymbol\xi ,
\end{split}
\label{S5}
\end{align}

\noi
where $\boldsymbol u = (u_1, \dots, u_m)$.
On the other hand, from \eqref{T2} with \eqref{S4},
 we have
\begin{align}
\begin{split}
&\Psi_N(x_2)^mS_\sigma(F_{1},\ldots,F_{m})(x_1,x_2)
\\
&=
\sum_{\substack{n_j\in\Lambda_j\\ j = 1, \dots, m}}
\int_{\R^{md_1}}\int_{\R^{md_2}}
 \sigma\left(
 \boldsymbol\xi,
 \boldsymbol n
 \right)
 \prod_{j=1}^m
 \widehat f_{j,n_j}(\xi_j)\widehat\Psi(u_j)
\\
&\hspace{2.8cm}\times
 e^{2\pi i x_1\cdot\sum_{j=1}^m\xi_j}
 e^{2\pi i x_2\cdot\sum_{j=1}^m n_j}
 e^{2\pi i (x_2/N)\cdot\sum_{j=1}^m u_j}
 \,d\boldsymbol u\,d\boldsymbol\xi.
\\
\end{split}
\label{S6}
\end{align}

\noi
Thus, from \eqref{eq:EN}, \eqref{S5}, and \eqref{S6}, we have
\begin{align}
\begin{split}
&  \lvert E_N(F_1,\ldots, F_m)(x_1,x_2)\rvert\\
 &\ \  \le
 \sum_{\substack{n_j\in\Lambda_j\\ j = 1, \dots, m}}
 \int_{\R^{md_1}}\int_{\R^{md_2}}
 \left|
 \sigma(\boldsymbol\xi,\boldsymbol n)
 -
 \sigma\left(
 \boldsymbol\xi,\boldsymbol n+\boldsymbol u/N
 \right)
 \right|
 \prod_{j=1}^m
|\widehat f_{j,n_j}(\xi_j)|
|\widehat\Psi(u_j)|
 \,d\boldsymbol u\,d\boldsymbol\xi.
\end{split}
\label{S7}
\end{align}

For every fixed \((\boldsymbol\xi,\boldsymbol n,\boldsymbol u)\), the continuity
of \(\sigma\) implies that
\begin{align}
2\norm{\sigma}_{L^\infty(\R^{md})}\geq \left|
 \sigma(\boldsymbol\xi,\boldsymbol n)
 -
 \sigma\left(
 \boldsymbol\xi,\boldsymbol n+\boldsymbol u/N
 \right)
 \right|\too 0,
\label{S8}
\end{align}

\noi
as $N \to \infty$.
By noting that
the sums over the \(n_j\)'s in \eqref{S8} are finite
and recalling that
\(\widehat f_{j,n_j}\in L^1 \left(\R^{d_1}\right)\), and
\(\widehat\Psi\in L^1\left(\R^{d_2}\right)\),
we see that
the right-hand side of
\eqref{S7} is bounded by an integrable function independent of
$N$, $x_1$, and $x_2$.  Therefore,
the convergence
\eqref{eq:ENuniform}
follows from \eqref{S8} and the Lebesgue dominated convergence theorem.

\medskip

\noindent
$\bul$~\textbf{Step 2:}
Since  \(\Psi_N=1\) on \(NQ\), the relation \eqref{eq:EN} simplifies to
\begin{align}
S_\sigma(F_1,\ldots,F_m)(x_1, x_2)
 =
 T_\sigma(F_{1,N},\ldots,F_{m,N})(x_1, x_2)+E_N(F_1,\ldots, F_m)(x_1, x_2)
\label{S9}
\end{align}

\noi
for $x_1\in\R^{d_1}$ and $x_2\in NQ$.

Fix $\lambda>0$
and define $\O_\ld$ by
\begin{align}
\Omega_\lambda=
 \big\{
 (x_1, x_2)\in\R^{d_1}\times Q:
 \lvert S_\sigma(F_1,\ldots,F_m)(x_1,x_2)\rvert>\lambda
 \big\}.
\label{S9a}
\end{align}

\noi
Since $S_\sigma(F_1,\ldots,F_m)$ is periodic in the second variable,
we have
\begin{align}
\lvert \Omega_\lambda\rvert=N^{-d_2}\bigl\lvert\bigl\{(x_1,x_2)\in\R^{d_1}\times NQ:\lvert S_\sigma(F_1,\ldots,F_m)(x_1,x_2)\rvert>\lambda\bigr\}\bigr\rvert.
\label{S10}
\end{align}

Given small  $\varepsilon\in (0, \lambda)$,
it follows from \eqref{eq:ENuniform} that
there exists $N_\eps \in \N$ such that
\begin{align}
 \norm{E_{N_\eps}(F_1,\ldots, F_m)}_{L^\infty(\R^d)}<\varepsilon
 \label{S11}
\end{align}

\noi
for  any $N \ge N_\eps$.
Hence, from \eqref{S9} and \eqref{S11}
 we have
\begin{align}
\begin{split}
&\big\{
 (x_1,x_2)\in\R^{d_1}\times NQ:
 \lvert S_\sigma(F_1,\ldots,F_m)(x_1,x_2)\rvert>\lambda
 \big\}
\\
&\quad\subset
\big\{
 x\in\R^d:
 \lvert T_\sigma(F_{1,N},\ldots,F_{m,N})(x)\rvert
 >\lambda-\varepsilon
 \big\}
 \end{split}
\label{eq:inclusion}
\end{align}

\noi
for any $N \ge N_\eps$.

Therefore,
from  \eqref{S10}, \eqref{eq:inclusion},
and  \eqref{eq:weakTbound}, we obtain
\begin{equation}
\label{eq:norms-1stesti}
 \lvert\Omega_\lambda\rvert
 \le
 N^{-d_2}
 \left(
 \frac{
 B\prod_{j=1}^m\norm{F_{j,N}}_{L^{p_j}(\R^d)}
 }{\lambda-\varepsilon}
 \right)^p
\end{equation}
for any $\varepsilon\in (0, \lambda)$ and $N \ge N_\eps$.

\medskip

\noindent
$\bul$~\textbf{Step 3:}
We now have a closer look at  the right-hand side of \eqref{eq:norms-1stesti}.
Recalling that
 $0\le\Psi_N\le1$ and \(\operatorname{supp}\Psi_N\subset NQ_K\),
 it follows from \eqref{S4} and the periodicity of $F_j$ in the second variable that
\begin{align*}
 \norm{F_{j,N}}_{L^{p_j}(\R^d)}^{p_j}&\leq
 \int_{\R^{d_1}}\int_{NQ_K}
 \lvert F_j(x_1,x_2)\rvert^{p_j}\,dx_2\,dx_1\\
 &\leq N^{d_2}\bigl(1+2/K+2/N\bigr)^{d_2}\int_{\R^{d_1}}\int_{Q}
 \lvert F_j(x_1,x_2)\rvert^{p_j}\,dx_2\,dx_1.
\end{align*}

\noi
Consequently, we have
\begin{equation}
\label{eq:cutoff-limsup}
 \limsup_{N\to\infty}
 N^{-d_2/p_j}
 \norm{F_{j,N}}_{L^{p_j}(\R^d)}
 \leq
 \left(1+\frac{2}{K}\right)^{d_2/p_j}
 \norm{F_j}_{L^{p_j}(\R^{d_1}\times\T^{d_2})}.
\end{equation}
Taking the product in \eqref{eq:cutoff-limsup} and using \eqref{T3},
we obtain
\begin{equation}
\label{eq:product-limsup}
 \limsup_{N\to\infty}
 N^{-d_2/p}
 \prod_{j=1}^m\norm{F_{j,N}}_{L^{p_j}(\R^d)}
 \leq
 \left(1+\frac{2}{K}\right)^{d_2/p}
 \prod_{j=1}^m
 \norm{F_j}_{L^{p_j}(\R^{d_1}\times\T^{d_2})}.
\end{equation}

\noi
Hence,
from \eqref{eq:norms-1stesti} and \eqref{eq:product-limsup}, we obtain
\begin{equation}
 \lvert\Omega_\lambda\rvert
 \leq
 \left(1+\frac{2}{K}\right)^{d_2}
 \left(\frac{B\prod_{j=1}^m
 \norm{F_j}_{L^{p_j}(\R^{d_1}\times\T^{d_2})}
 }{\lambda-\varepsilon}\right)^p.
\label{S12}
\end{equation}
Finally, by taking $K\to\infty$ and $\varepsilon\to 0$ in \eqref{S12},
we obtain
\begin{align}
\lvert\Omega_\lambda\rvert\leq\left(\frac{B\prod_{j=1}^m
 \norm{F_j}_{L^{p_j}(\R^{d_1}\times\T^{d_2})}}{\lambda}\right)^p.
\label{S13}
\end{align}

\noi
Therefore, the bound
\eqref{S1}
follows from
\eqref{S9a}, \eqref{S13}, and \eqref{S2}
with the density argument.
\end{proof}

\section{Applications}\label{SEC:Applications}

Recall the notation from Section~\ref{SEC:Multilinear}, in particular, the definitions of the classical
$m$-linear multiplier operator $T_\sigma$ in \eqref{T1} and of
the $m$-linear multiplier operator $S_\sigma$ on a product space defined in \eqref{T2}. A classical result due to Coifman and Meyer (\cite[Proposition~2 on  p.\,154]{CM2}; see also \cite[Proposition 6]{GT02}) states that if $\sigma\in C^\infty(\R^{md})$ and
\begin{equation}\label{eq:CM-multi}
|\partial_{\boldsymbol\xi}^{\boldsymbol\beta}\partial_{\boldsymbol\eta}^{\boldsymbol\gamma}\sigma
(\boldsymbol\xi, \boldsymbol\eta)|\lesssim
\big(1+ |\boldsymbol\xi|+|\boldsymbol\eta|\big)^{-|\boldsymbol\beta|-|\boldsymbol\gamma|}
\end{equation}
for all $\boldsymbol\xi\in\R^{md_1}, \boldsymbol\eta\in\R^{md_2}$ and all multi-indices
$\boldsymbol\beta\in\N_0^{md_1}$ and $ \boldsymbol\gamma\in\N_0^{md_2}$, then the $m$-linear multiplier operator $T_\sigma$ is bounded from  $L^{p_1}(\R^d)\times\cdots\times L^{p_m}(\R^d)$
to $ L^p(\R^d)$
for any  $1<p_1,\ldots, p_m<\infty$ satisfying  \eqref{T3}. 
Moreover, if at least one of the $p_j$'s equals $1$, then $T_\sigma$
is bounded from  $L^{p_1}(\R^d)\times\cdots\times L^{p_m}(\R^d)$
to the weak space  $L^{p, \infty}(\R^d)$; in particular,
we have $T_\sigma: L^{1}(\R^d)\times\cdots\times L^{1}(\R^d)\to L^{1/m, \infty}(\R^d)$.

Consequently, by using Propositions~\ref{LEM:multitrans2} and~\ref{LEM:multitrans3}, we obtain that for any multiplier $\sigma\in C^\infty(\R^{md})$, satisfying \eqref{eq:CM-multi}, and $1<p_1, \ldots, p_m<\infty$, satisfying \eqref{T3},   we have
\begin{align}
S_\sigma: L^{p_1}(\R^{d_1}\times\T^{d_2})\times\cdots\times
 L^{p_m}(\R^{d_1}\times\T^{d_2})\to
 L^p(\R^{d_1}\times\T^{d_2}),
\label{X0a}
\end{align}

\noi
while if at least one of the $p_j$'s equals $1$, then we have
\begin{align}
S_\sigma: L^{p_1}(\R^{d_1}\times\T^{d_2})\times\cdots\times
 L^{p_m}(\R^{d_1}\times\T^{d_2})\to
 L^{p, \infty}(\R^{d_1}\times\T^{d_2}).
\label{X0b}
\end{align}

\subsection{Multilinear Riesz-type and tensor-like operators}\label{SUBSEC:bilin}
In the simplest case $d_1=d_2=1$, let
$\langle (\boldsymbol\xi, \boldsymbol\eta)\rangle=\big(1+4\pi^2 (|\boldsymbol\xi|^2+|\boldsymbol\eta|^2)\big)^{1/2}$ denote the symbol of the Bessel potential operator $(\Id-\Delta)^{1/2}.$  It is easy to see that
the multipliers
\begin{align}
\sigma_{1, j}(\boldsymbol\xi, \boldsymbol\eta)=
\frac{\xi_j}{\langle (\boldsymbol\xi, \boldsymbol\eta)\rangle}
\qquad\text{and}\qquad
\sigma_{2, j}(\boldsymbol\xi, \boldsymbol\eta)
=\frac{\eta_j}{\langle (\boldsymbol\xi, \boldsymbol\eta)\rangle}
\label{X1}
\end{align}

\noi
both
satisfy \eqref{eq:CM-multi}, $j = 1, \dots m$.

 We now give a semigroup interpretation to the
 multilinear Riesz-type  operators $S_{\sigma_{k, j}}$, $k = 1, 2$ and $j = 1, \dots, m$,
 where $S_\s$ is as in \eqref{T2}.
Recalling that $\G(1/2) = \sqrt \pi$, we have
\[
 A^{-1/2}=\frac1{\sqrt{\pi}}\int_0^\infty t^{-1/2}e^{-tA}\,dt,
 \qquad A>0.
\]

\noi
Applying this identity with
 \(A=\langle(\boldsymbol\xi,\boldsymbol n)\rangle^2\) yields
\begin{equation}
 \frac1{\langle(\boldsymbol\xi,\boldsymbol n)\rangle}=
 \frac1{\sqrt{\pi}}\int_0^\infty t^{-1/2}e^{-t}\prod_{j=1}^m
 e^{-4\pi^2 t(\xi_j^2+n_j^2)}
 \,dt.
 \label{X2}
\end{equation}

\noi
Let \(P_t\)  be the heat semigroup on \(\R\times\T\), given by
\[
 \widehat{P_t F}(\xi,n)=e^{-4\pi^2  t(\xi^2+n^2)}\widehat F(\xi,n), \quad (\xi,n)\in\R\times\Z.
\]

\noi
Also, for $k = 1, 2$,  set
$D_k=1/(2\pi i)\partial_{x_k}$.
Then,  a routine calculation
with \eqref{T2}, \eqref{X1}, and \eqref{X2}
shows that
\begin{align*}
& S_{\sigma_{k, j}}(F_1,\ldots,F_m)(x_1, x_2)\\
& \quad=
\frac1{\sqrt{\pi}}\int_0^\infty t^{-1/2}e^{-t}\bigl(D_k P_t F_j\bigr)(x_1,x_2)
\prod_{\substack{\l=1\\ \l\neq j}}^m
\bigl(P_t F_\l\bigr)(x_1,x_2)
\,dt
\end{align*}

\noi
Equivalently,
denoting by \(K_t\)  the convolution kernel of \(P_t\) on
\(\R\times\T\), the preceding formula can be rewritten as
\begin{equation}
\label{eq:mixed-heat}
S_{\sigma_{k, j}}(F_1,\ldots,F_m)=
\frac1{\sqrt{\pi}}\int_0^\infty t^{-1/2}e^{-t} \bigl(D_k K_t*F_j\bigr)
\prod_{\substack{\l=1\\ \l\neq j}}^m
\bigl(K_t*F_\l\bigr)
\,dt,
\end{equation}

\noi
providing the heat-semigroup representation of the multilinear Riesz-type transforms
 $S_{\sigma_{k, j}}$, $k = 1, 2$ and $j = 1, \dots, m$.
This representation is reminiscent of the classical heat-semigroup representation of the Riesz transforms
$R_j$:
\[R_j f = \frac 1 {\sqrt{\pi}} \int_0^\infty t^{-1/2}\dd_j e^{t\Dl} f \, dt, \]

\noi
which follows from the Laplace-transform representation of $(-\Delta)^{-1/2}$.
%
As a consequence of Propositions~\ref{LEM:multitrans2} and~\ref{LEM:multitrans3},
these Riesz-type transforms $S_{\s_{k, j}}$ given in \eqref{eq:mixed-heat} are bounded from $L^{p_1}(\R\times\T)\times\cdots\times
 L^{p_m}(\R\times\T)$ to $L^p(\R\times\T)$ for $1<p_\l<\infty$, satisfying~\eqref{T3},
and are also bounded from
$L^{p_1}(\R \times\T)\times\cdots\times
 L^{p_m}(\R\times\T)$ to $L^{p, \infty}(\R\times\T)$ if one of the $p_\l$'s equals $1$.

A recurrent method of obtaining estimates in PDEs is
to truncate a given multiplier at appropriate scales; see \cite[Example 6]{BO2013}.
For simplicity,
we restrict our attention to the linear case ($m=1$)
and  $d_1=d_2=1$.
Given
 smooth cut-off functions
 $\phi_1, \phi_2$,
 supported on annuli away from the origin,  and coefficients $\{c_{a, b}\}_{a, b\in \Z}\in\C$,
satisfying
$\sup_{a,b\in\Z}|c_{a,b}|<\infty$,
 consider
 the multiplier
\begin{align}
 \s(\xi, n) = \sum_{a, b\in\Z}c_{a, b}\, \phi_1(2^{-a}\xi)\phi_2(2^{-b}n).
\label{X3}
\end{align}

\noi
Note that, for fixed $\xi, n$, the summation in $a, b$ is finite.
The multiplier $\s$ in \eqref{X3} is generally not a Coifman--Meyer multiplier in the isotropic sense of
\eqref{eq:CM-multi}, since the two frequency variables may occur at very different scales. Instead, it is naturally a multiplier of Marcinkiewicz type. Indeed, by the support properties of \(\phi_1\) and \(\phi_2\), for any $k, \l \in \N_0$,  we have
\begin{align*}
    \big|
    \partial_\xi^k\partial_\eta^\ell\sigma(\xi,\eta)
    \big|
    \lesssim_{k,\ell}
    \|\{c_{a,b}\}\|_{\ell^\infty_{a,b}}
    |\xi|^{-k}|\eta|^{-\ell},
    \qquad \xi\eta\ne0.
\end{align*}

\noi
 Thus, the classical Marcinkiewicz multiplier theorem on \(\mathbb R^2\)
(\cite[Theorem 6.2.4 and Corollary 6.2.5]{Gra1}) yields
\begin{align}
\|T_\sigma f\|_{L^p(\mathbb R^2)}
\lesssim_p
\|\{c_{a,b}\}\|_{\ell^\infty_{a, b}}
\|f\|_{L^p(\mathbb R^2)},
\qquad 1<p<\infty.
\label{X3a}
\end{align}

We may now apply Proposition~\ref{LEM:multitrans2} to transfer this Euclidean estimate to the product space setting. More precisely, let \(S_\sigma\) be the multiplier operator on \(\mathbb R\times\mathbb T\) given by

$$
S_\sigma(F)(x_1,x_2)
=
\sum_{n\in\mathbb Z}
\int_{\mathbb R}
\left[
\sum_{a,b\in\mathbb Z}
c_{a,b}\phi_1(2^{-a}\xi)\phi_2(2^{-b}n)
\right]
\widehat F(\xi,n)
e^{2\pi i x_1\xi}e^{2\pi i x_2n}\,d\xi .
$$

\noi
Given \(n\in\mathbb Z\), set
$$
\sigma_n(\xi)=\sigma(\xi,n)
\qquad\text{and}\qquad
F_n(x_1)
=
\int_{\mathbb T}
F(x_1,x_2)e^{-2\pi inx_2}\,dx_2.
$$

\noi
Then \(\widehat F_n (\xi)=\widehat F(\xi,n)\), and
$$
S_\sigma(F)
=
\sum_{n\in\mathbb Z}
T_{\sigma_n}(F_n)\otimes e_n^{(2)},
$$
where \(T_{\sigma_n}\) denotes the usual Fourier multiplier on \(\mathbb R\) with symbol \(\sigma_n\).

Consequently, Proposition~\ref{LEM:multitrans2} and the Marcinkiewicz multiplier theorem give
$$
\|S_\sigma F\|_{L^p(\mathbb R\times\mathbb T)}
\le
\|T_\sigma\|_{L^p(\mathbb R^2)\to L^p(\mathbb R^2)}
\|F\|_{L^p(\mathbb R\times\mathbb T)}
\lesssim_p
\|\{c_{a,b}\}\|_{\ell^\infty}
\|F\|_{L^p(\mathbb R\times\mathbb T)}
$$

\noi
for every \(1<p<\infty\). This example illustrates that the transference principle of Proposition~\ref{LEM:multitrans2} is not restricted to Coifman--Meyer multipliers;
 Euclidean multiplier estimates arising from genuinely product-type frequency decompositions can also be transferred directly to product spaces.

\subsection{Multipliers from PDEs}\label{SUBSEC:cone}
One natural setting where product spaces
arise is that of an evolution equation posed on a product spatial domain $\R^{d_1} \times \T^{d_2}$,
where the Coifman--Meyer boundedness results \eqref{X0a} and \eqref{X0b}
may be of use.
Another example appears
in the study of an evolution equation posed on a periodic spatial domain;
 the time variable is Euclidean, whereas the spatial variables are periodic.  We provide here two simple illustrations: the propagator of the damped Klein-Gordon equation and a smooth localization to the characteristic cone of the wave operator. These examples are inspired by the recent work~\cite{Zine}, where the stochastic damped sine--Gordon equation was studied on \(\R_+\times\T^2\), and the deterministic analysis involves cone multipliers
and the hyperbolic Riesz potential.

The damped Klein--Gordon propagator is given by
\[
\mathcal D(t)
=
e^{-t/2}
\frac{\sin\bigl(t[\![\nabla]\!]\bigr)}
     {[\![\nabla]\!]},
\qquad
[\![n]\!]
=
\left(\frac{3}{4}+4\pi^2 |n|^2\right)^{1/2};
\]
see \cite[(1.23)]{Zine}. Thus, its multiplier on $\R\times\R $ is given by
\[
\sigma_t^{\text{KG}}(\xi,\eta)
=
te^{-t/2}\text{sinc}\left(
t\sqrt{\frac{3}{4}+4\pi^2(\xi^2+\eta^2)}
\right),
\]

\noi
where $\text{sinc}(x) = \sin x /x$.
The corresponding mixed multiplier operator given by
\[
S_{\sigma_t}^{\text{KG}}F(x_1, x_2)=
\sum_{n\in\mathbb Z}\int_{\mathbb R}
\sigma_t^{\text{KG}}(\xi,n)\widehat F(\xi, n)e^{2\pi i(x_1\xi+x_2 n)}
\,d\xi
\]
is the  propagator associated with the damped Klein--Gordon equation
\[
\partial_t^2 u+\partial_t u+
(\text{Id}-\Delta_{\mathbb R\times\mathbb T})u=0
\]
on the product space $\mathbb R\times\mathbb T$.
We note that, with the notation from Section~\ref{SEC:Multilinear}, the two-dimensional
damped Klein-Gordon propagator has an integrable physical-space kernel on $\R^2$.
Thus, from  Proposition~\ref{LEM:multitrans1},
 we obtain the boundedness of $S_{\sigma_t}^\text{KG}$ on all $L^p(\R\times\T)$, $1\leq p<\infty$.

Our next illustration concerns a smooth cone-frequency localization for the wave
equation on \(\R\times\T^2\). The d'Alembertian
\(\Box=\partial_t^2-\Delta_x\)
has symbol
\(-4\pi^2(\tau^2-|n|^2),\)
so its characteristic set is given by
\[
 \{(\tau, n)\in\R\times\Z^2:\tau^2=|n|^2\}.
\]
It is thus natural in this context to consider first the \emph{sharp} cone
multiplier:
\begin{equation}
\label{eq:sharp-cone}
 \widehat{\mathcal C f}(\tau,\xi)
 =
 \mathbf 1_{\{|\tau|>|\xi|\}}
 \widehat f(\tau, \xi).
\end{equation}
However, the sharp multiplier $\mathcal C$ defined in \eqref{eq:sharp-cone} is not bounded on
\(L^p(\R\times\R^2)\) for \(1<p<\infty\), \(p\neq2\),
as a consequence of Fefferman's ball multiplier theorem
\cite{Fefferman}; see also \cite{Mock, LV}.

Below we introduce a smooth cone localization which is equally natural in the analysis of the regions close to the characteristic set
\(\{(\tau, n)\in\R\times\Z^2:|\tau|=|n|\}.\) More specifically, let \(\psi\in C_c^\infty(\R; [0, 1])\) satisfy
\[
 \psi(s)=1\  \ \text{for }s\in [-1, 1]
 \qquad\text{and}\qquad
 \operatorname{supp}\psi\subset [-2,2].
\]

\noi
Given  \(0<\delta<1\), define the symbol $\s_\dl$ by
\begin{equation}
\label{eq:smooth-cone-symbol}
\sigma_\delta(\tau,\xi)=\psi(\delta^{-1}r(\tau, \xi)),\qquad r(\tau, \xi)=
 \frac{\tau^2-|\xi|^2}{1+\tau^2+|\xi|^2}.
\end{equation}

\noi
for $(\tau, \xi)\in\R\times\R^2$.
On the support of \(\sigma_\delta\), we have
\(|\tau^2-|\xi|^2|\lesssim \delta(1+\tau^2+|\xi|^2)
\)
while  \(1\) in the denominator of \eqref{eq:smooth-cone-symbol} regularizes the symbol near the origin.
We define the smooth cone-frequency localization for the wave equation on \(\R\times\T^2\)
 by the following multiplier operator:
\begin{equation}
\label{eq:mixed-cone-operator}
 S_{\sigma_\delta}F(t, x)=\sum_{n\in\Z^2}\int_\R
 \sigma_\delta(\tau, n)\widehat F(\tau, n)
 e^{2\pi i(t\tau+x\cdot n)}\,d\tau.
\end{equation}

 It is easy to verify that for every multi-index
\(\alpha=(\alpha_0,\alpha_1,\alpha_2)\in\mathbb N_0^3\),
we have
\begin{equation*}
 \big|
 \partial_\tau^{\alpha_0}
 \partial_\xi^{(\alpha_1,\alpha_2)}
 r(\tau,\xi)
 \big|
 \leq
 C_\alpha
 (1+|\tau|+|\xi|)^{-|\alpha|}
\end{equation*}

\noi
Namely,  $\sigma_\delta$ satisfies \eqref{eq:CM-multi}.
The Euclidean \(L^p\) and weak-\(L^1\) estimates for the
smooth cone-frequency localization on $\R\times \R^2$  then transfer  to the physically
natural space-time domain \(\R\times\T^2\).
Namely,  Proposition~\ref{LEM:multitrans2}
and~\ref{LEM:multitrans3}
yield
  the following bound for the operator $S_{\s_\dl}$ in~\eqref{eq:mixed-cone-operator}:
\begin{align*}
\norm{S_{\sigma_\delta}F}_{L^p(\R\times\T^2)}
&  \leq
 \norm{T_{\sigma_\delta}}_{L^p(\R^3)\to L^p(\R^3)}
 \norm{F}_{L^p(\R\times\T^2)},
 \qquad 1<p<\infty,\\
 \norm{S_{\sigma_\delta}F}_{L^{1,\infty}(\R\times\T^2)}
&  \leq
 \norm{T_{\sigma_\delta}}_{L^1(\R^3)\to L^{1,\infty}(\R^3)}
 \norm{F}_{L^1(\R\times\T^2)}.
\end{align*}

\medskip

\begin{ackno}\rm
T.O.~was supported by the European Research Council (grant no.~864138 ``SingStochDispDyn")
and also
acknowledges support from
the NSFC (grant no.~W2531005).
\end{ackno}

\end{document}